\RequirePackage{fix-cm}
\documentclass[smallcondensed]{svjour3} 
\smartqed 
\usepackage[rightcaption]{sidecap}

\usepackage{amsmath}
\usepackage{tcolorbox}
\usepackage{algorithm}

\usepackage[framemethod=TikZ]{mdframed}
\usepackage{lipsum}
\mdfdefinestyle{MyFrame}{%
	linecolor=black,
	outerlinewidth=0.005cm,
	roundcorner=18pt,
	innertopmargin=\baselineskip,
	innerbottommargin=\baselineskip,
	backgroundcolor=isabelline}

 \usepackage{algpseudocode}
\definecolor{isabelline}{rgb}{0.96, 0.94, 0.93}

 \usepackage{graphicx, mathptmx, latexsym, url} 
\usepackage{amssymb, amsmath, bm, color, threeparttable, dcolumn, bbding, booktabs}
\usepackage[colorlinks=true,linkcolor=blue,citecolor=blue,filecolor=blue]{hyperref}
\usepackage{natbib, lineno, animate, pdflscape, rotating}%
\usepackage{tikz}
\usepackage{enumitem}
\makeatletter
\newcommand*\qCircle[2][1.6]
{\tikz[baseline=(char.base)]{
\node[shape=circle, draw, inner
sep=1pt,
minimum height={\f@size*#1},]
(char){\vphantom{QAHER}#2};}}
\makeatother
\usepackage{tikz} 

\newcommand*\tddd{\tikz[baseline=(char.base)]{ 
\node[shape=circle,draw,inner sep=0.09pt] (char) {$\dagger$};}}
\newcommand{\qlestardot}{\overset{c, \dagger}{\leqslant}}
\newcommand{\qlestardod}{\overset{d, \dagger}{\leqslant}}

\newcommand{\qlestardooD}{\overset{d}{\preceq}} 
\newcommand{\qlestardodd}{\overset{\dagger ,d}{\leqslant}}
\newcommand{\nqlestardot}{\overset{c, \dagger}{\nleqslant}}

\newcommand{\nqlestardodd}{\overset{\dagger ,d}{\nleqslant}}

 \newcommand\CG[1]{\textcolor{black}{#1}}

\definecolor{folly}{rgb}{1.0, 0.0, 0.31}

\begin{document}

\title{  Properties of core-EP matrices and binary relationships }
	\titlerunning{Properties of core-EP matrices and binary relationships}

	\author{    
    \mbox {Ehsan Kheirandish \and Abbas Salemi \and N\'estor Thome} 
	}
	\authorrunning{\CG{E. Kheirandish, A. Salemi, N. Thome}} 
	\institute{This paper is dedicated to Professor Chi-Kwong Li in honor of his 65th birthday and in recognition of his substantial contributions to Linear Algebra, Operator Theory, and their applications. 
 \makebox[\textwidth][c]{\hrulefill}\newpage 
		Ehsan Kheirandish 
		\at Department of Applied Mathematics,  Shahid Bahonar University of Kerman, Kerman, Iran\\
		\email{ehsankheirandish@math.uk.ac.ir}
		\and
		Abbas Salemi     
		\at Department of Applied Mathematics,  Shahid Bahonar University of Kerman, Kerman, Iran\\
		\email{salemi@uk.ac.ir}
		\and		
		N\'estor Thome
	\at Instituto Universitario de Matem\'atica Multidisciplinar, Universitat Polit\`ecnica de Val\`encia, Valencia, 46022, Spain  \\
		\email{njthome@mat.upv.es},
}
	\date{}
	\maketitle
	\begin{abstract}
 In this paper,  various properties of core-EP matrices are investigated. We introduce the MPDMP matrix associated with $A$  and by means of it,  some properties and equivalent conditions of core-EP matrices can be obtained. Also, properties of  MPD, DMP, and CMP  inverses are studied and we prove that in the class of core-EP matrices, DMP, MPD, and Drazin inverses are the same.
Moreover,  DMP and MPD binary relation orders are introduced and the relationship between these orders and other binary relation orders are considered.

	\end{abstract}
	\section{Introduction}
The term core-EP matrix was introduced in \cite{benitez2010matrices} to emphasize its connection with the class of matrices known as EP matrices, alternatively referred to as range-Hermitian matrices. 
This specific category has received considerable interest over time due to its fascinating properties.
The investigation presented in \cite{further},  focuses on core-EP matrices, identifying several unique features within this class and delineating new characteristics.
 Moreover, characterizations and applications  of core-EP decomposition  can be found in \cite{wang2016, benitez2012canonical, Ferreyra}.  Some more characterizations of  the core-EP inverse and applications 
the perturbation bounds related to the core-EP inverse  and upper bounds for the errors 
$\Vert B^{\tddd}-A^{\tddd}\Vert /\Vert A^{\tddd}\Vert$ and $\Vert BB^{\tddd}-AA^{\tddd}\Vert$
 can be found in literatures \cite{zhou2021characterizations, ma2019characterizations, TOYUER, ma2021characterizations, mosic2021generalization}.

A partial order on a nonempty set is defined as a binary relation that meets the criteria of reflexivity, transitivity, and antisymmetry. Recently, there has been a growing interest among mathematicians in the field of matrix partial ordering, \cite{zhang2023new}.
The applications of generalized inverses extend to various domains including mathematics, channel coding and decoding, navigation signals, machine learning, data storage, and cryptography. Specifically, systematic non-square binary matrices, such as the $(n-k)\times k$ matrix $H$ where $n > k,$ play a crucial role. For this type of matrix, there exist precisely $2k\times (n-k)$ distinct generalized inverse matrices. These matrices find utility in cryptographic systems like the McEliece and Niederreiter public-key cryptosystems \cite{makoui20}, massively parallel systems \cite{stanojevic20}, neural network \cite{xing20}, Engineering \cite{ansari20}, machine learning \cite{kim20}, computation vision \cite{brockett19}, and data mining \cite{lash20}.

In this paper, the set of all $ m \times n $ complex matrices is represented by $M_{m,n}(\mathbb{C})$.
When $m=n$, we will simply write $M_n({\mathbb C})$ instead of $M_{n,n}({\mathbb C})$.  Let $A\in M_{m,n}(\mathbb{C})$,
the symbols $A^{*}$, $\mathcal{R}(A)$, $\mathcal{N}(A)$, and rank($A$) will stand for conjugate transpose, column space, 
null space, and rank of the matrix $ A$, respectively.

Given $ A \in M_{n}(\mathbb{C}), $ the index of $A$ (represented by Ind($A$)) is the smallest non-negative integer $k$ such that
rank($A^{k}$)=rank($A^{k+1}$), and the \emph{Drazin inverse} of  $A$ is the unique solution  
that satisfies 
\begin{equation*}
A^{k+1}X=A^{k}\quad \& \quad XAX=X\quad \& \quad AX=XA,
\end{equation*}
where $k$=Ind($A$), and it is represented by $ A^{d}$. If Ind($A$)$\leqslant 1$, then $A^d$ is called the \emph{group inverse} of $A$ and denoted by $A^\#$.

For $ A \in M_{m,n}(\mathbb{C}) $, if $ X \in M_{n,m}(\mathbb{C}) $ satisfies
\begin{equation}\label{MMoo}
AXA=A \quad \&\quad XAX=X \quad \&\quad (AX)^{*}=XA \quad \& \quad (XA)^{*}=XA,
\end{equation}
then $X$ is called a \emph{Moore-Penrose inverse} of $A$ and this  matrix $X$ is unique and represented by $A^{\dagger}$.

Note that $P_{A}:=AA^{\dagger} $ and $    Q_{A}:=A^{\dagger}A$ are the orthogonal projectors onto 
$\mathcal{R}(A)  $ and $  \mathcal{R}(A^{*}) ,$ respectively (see \cite{Ben, ilic2017algebraic, wang2018generalized, Ben1}).

Let $ A \in M_{n}(\mathbb{C}).$   From \cite[Theorem 2.2.21]{MitraSK}, there are unique matrices $A_{c}$ and $A_{n}$ such that 
$\rm{Ind}(A_{c}) \leqslant 1$ and $A_{n}$ is a nilpotent matrix satisfying $A = A_{c} +  A_{n}$ and $A_{c} A_{n} =  A_{n}A_{c} = 0$.
 The matrix $ A_{c} $ is known as \emph{core part} of $A$ and $ A_{n} $ is known as the \emph{nilpotent part}. 
A matrix  $ A \in M_{n}(\mathbb{C})$ is called \emph{core-EP} if $ A^{\dagger}A_{c}=A_{c}A^{\dagger}$ (see \cite{am16}). It is interesting note that every EP-matrix is a core-EP matrix but the converse, in general, is not true. 

Let $ A \in M_{n}(\mathbb{C}).$ The unique matrix $ X \in M_{n}(\mathbb{C}) $ satisfying 
\begin{equation*}
XAX=X\quad  \&\quad  \mathcal{R}(X)=\mathcal{R}(X^{*})=\mathcal{R}(A^{k}),
\end{equation*}
is called  the \emph{core-EP inverse} of the matrix $A$ and is represented by $ A^{\tddd} $ (see \cite{core}).
The unique matrix $ X \in M_{n}(\mathbb{C}) $ satisfying 
\begin{equation*}
XAX=X\quad \&\quad XA=A^{d}A\quad \& \quad A^{k}X=A^{k}A^{\dagger},
\end{equation*}
is known as the \emph{DMP-inverse} of $A$ and is represented by $ A^{d ,\dagger}$.
Moreover,  the DMP-inverse can be represented as $ A^{d, \dagger}=A^{d}AA^{\dagger}$  (see \cite{Malik}).

The \emph{CMP-inverse} of $ A \in M_{n}(\mathbb{C}) $ was defined as $ A^{c,\dagger}=A^{\dagger}A_{c}A^{\dagger}$ in \cite{am16}. In \cite{KS}, the \emph{CMP-inverse} was improved as the unique solution of the  following equations:
\begin{equation*}
XAX=X\quad \&\quad AX=A_{c}A^{\dagger}\quad \&\quad XA=A^{\dagger}A_{c}.
\end{equation*}
An application of DMP and CMP inverses to tensor can be found in \cite{kheirandish2023generalized}.
Some more properties of these generalized inverses and applications can be found in 
literatures  \cite{Ma, CvMoWe,LiJiCv, ma2020characterizations, wang2024dual}.

In \cite[Corollary 6]{Hartwig}, it was proved that every matrix $ A \in M_{n}(\mathbb{C}) $ with \\rank$(A)=r >0$, has a \emph{Hartwig-Spindelb{\"o}ck 
decomposition:} 
\begin{equation}\label{h}
A = U \left( \begin{array}{cc}
\Sigma Q & \Sigma P \\ 
0 & 0
\end{array}\right) U^{*}, 
\end{equation}
where $ U \in M_{n}(\mathbb{C}) $ is a unitary matrix, $ \Sigma = diag(\sigma_{1}I_{k_{1}}, \sigma_{2}I_{k_{2}}, \ldots, \sigma_{t}I_{k_{t}}) $ 
is a diagonal matrix, the entries on the diagonal $ \sigma_{j} > 0 \quad ( j=1,\cdots ,t )$ being the singular values of the matrix $A$,
$ \sum _{j=1}^t k_{j}=r$,
$ Q \in M_{r}(\mathbb{C}) $ and $ P \in M_{r ,n-r}(\mathbb{C}) $ satisfy 
\begin{equation}\label{KL}
QQ^{*}+ PP^{*} = I_{r}.
\end{equation}
In \cite{Malik}, we can found that
\begin{equation}\label{md}
\begin{small}
A^{\dagger} = U \left( \begin{array}{cc}
Q^{*}\Sigma ^{-1} & 0 \\ 
P^{*}\Sigma ^{-1} & 0
\end{array} \right) U^{*}
\end{small}
\quad \&\quad 
\begin{small}
A^{d} = U \left( \begin{array}{cc}
(\Sigma Q)^{d} & ((\Sigma Q)^{d})^{2}\Sigma P \\ 
0 & 0
\end{array} \right) U^{*}.
\end{small}
\end{equation}
By using $A_{c}=A A^{d}A,$ we have
\begin{equation}\label{cm}
A_{c} = U \left( \begin{array}{cc}
\Sigma \hat{Q} \Sigma Q & \Sigma \hat{Q} \Sigma P \\ 
0 & 0
\end{array} \right) U^{*},
\end{equation} 
where $ \hat{Q}=Q(\Sigma Q)^{d}$. 

The main aim of this paper is to introduce a new matrix (named MPDMP matrix) associated with a given matrix. We found that the MPDMP matrix has interesting properties (for example, see Theorem \ref{c*}, Remark \ref{KheiEHSA78}, and Theorem \ref{ER}). Also, some properties and equivalent conditions of core-EP matrices can be obtained by MPDMP matrices (see Corollary \ref{KheiEHSA7889}).

This paper is organized as follows. Section 2 introduces  the MPDMP matrix associated with $A$, and is devoted to obtaining properties and equivalent conditions of core-EP matrices.  Moreover, we get equivalent conditions for $A^{\dagger ,d,\dagger}, A^{C,\tddd}$ and $  A^{c,\dagger}$ to be an EP matrix. In Section 3,
 new properties of known generalized inverses will be considered. In addition, some properties of DMP and  CMP inverses are studied.
 In Section 4,  DMP and MPD binary relations are defined and their relationship with other binary relations is investigated.
\section{Properties of core-EP matrices}
In this section, the Moore-Penrose-Drazin-Moore-Penrose (MPDMP)  matrix associated with $A$ is introduced, and by using this definition, 
some properties and equivalent conditions of core-EP matrices are presented.
\begin{theorem}
Let $ A \in M_{n}(\mathbb{C})$. 
Then $ X= A^{d}A^{\dagger}$ is the unique solution of the following equations:
\begin{equation}\label{a2}
XP_A=X\quad \& \quad XA=A^{d}.
\end{equation}
Analogously, the unique matrix that satisfies 
\begin{equation*}
Q_AX=X\quad \& \quad AX=A^{d}.
\end{equation*}
is given by $X=A^{\dagger}A^{d}$. 
\end{theorem}
\begin{proof}
It is evident that the matrix $ X=A^{d}A^{\dagger} $ fulfills the two equations in the system (\ref{a2}).
Now, we consider that matrices $X_{1}$ and $ X_{2} $ satisfy (\ref{a2}). Then
\begin{align*}
X_{1}=X_1P_A=X_1AA^{\dagger}=A^{d}A^{\dagger}=X_2AA^{\dagger}=X_2P_A=X_2.
\end{align*}
The case for $X=A^{\dagger}A^{d}$  can be proven in a similar manner.
\end{proof}
\begin{theorem}\label{c*} 
Let $ A \in M_{n}(\mathbb{C}). $ Then $ X= A^{\dagger}A^{d}A^{\dagger}$ is the unique solution of the following equations:
\begin{equation}\label{a1}
XA^{3}X=X\quad \& \quad AX=A^{d}A^{\dagger} \quad \& \quad XA=A^{\dagger }A^{d}.
\end{equation}
\end{theorem}
\begin{proof}
It is evident that the matrix $ X=A^{\dagger}A^{d}A^{\dagger} $ fulfills the three equations in the system (\ref{a1}).
Now, we consider that matrices $X_{1}$ and $ X_{2} $ satisfy (\ref{a1}). Then
\begin{align*}
X_{1}&=X_{1}A^{3}X_{1}=X_1AA^{2}X_1=A^{\dagger}A^{d}AAX_1=A^{\dagger}A^{d}AA^{d}A^{\dagger}\\
&=X_2AAA^{d}A^{\dagger}=X_2AAAX_2=X_2A^{3}X_2=X_2.
\end{align*}
\end{proof}
Now, we define the MPDMP matrix associated with $A$.
\begin{definition}\label{ASER}
Let $ A \in M_{n}(\mathbb{C}). $ 
The Moore-Penrose-Drazin-Moore-Penrose (MPDMP)  matrix associated with $A$ is defined and denoted by 
\begin{equation*}
A^{\dagger ,d,\dagger} :=A^{\dagger}A^{d}A^{\dagger}.
\end{equation*}
\end{definition}
\begin{remark}\label{KheiEHSA78}
Let $ A \in M_{n}(\mathbb{C})$. 
Then a similar approach as in the proof of Theorem \ref{c*}, the following systems of equations  are consistent and they have the unique solution 
$ X= A^{\dagger}A^{d}A^{\dagger}$.
\begin{eqnarray*} 
(i) &Q_{A}XP_{A}=X, ~~~~~&\&~~~~~ AX=A^{d}A^{\dagger}.\\
(ii)&Q_{A}XP_{A}=X, ~~~~~&\&~~~~~ AXA=A^{d}.\\
(iii)& Q_{A}XP_{A}=X, ~~~~&\&~~~~~ XA=A^{\dagger}A^{d}.\\
(iv)& XP_{A}=X, ~~~~~&\&~~~~~ XA=A^{\dagger }A^{d}.\\
(v) &Q_{A}X=X, ~~~~~&\&~~~~~ AX=A^{d}A^{\dagger}.
\end{eqnarray*}
\end{remark}
\begin{theorem}\label{as20}
Let $ A \in M_{n}(\mathbb{C})$ be 
a core-EP matrix. The  following conditions are equivalent: 
\begin{enumerate}
\item the DMP inverse is equal to the Drazin inverse,
\item the MPD (the dual of DMP \cite [Remark 2.9]{Malik}) inverse is equal to the Drazin inverse. 
\end{enumerate}
\end{theorem}
\begin{proof}
Suppose that $A$ is a core-EP matrix. We have
\begin{align*}
A^{d,\dagger}=A^d&\Leftrightarrow A^{d}AA^{\dagger}=A^d\Leftrightarrow (A A^d A) A^{\dagger} =A A^d\Leftrightarrow A_{c}A^{\dagger}=AA^d\\
&\Leftrightarrow A^{\dagger}A_{c}=AA^d
\Leftrightarrow A^{\dagger}AA^{d}A=AA^d
\Leftrightarrow A^{\dagger}AA^{d}=A^d
\Leftrightarrow A^{\dagger ,d}=A^{d}.
\end{align*}
\end{proof}
We have proved that in the class of core-EP matrices, DMP, MPD, and Drazin inverses are the same.
The hypothesis of core-EPness in Theorem \ref{as20} is essential. 
\begin{example}\label{EXAMV5}
Let
$
A = \left( \begin{array}{ccc}
2 & 0 & 1 \\ 
0 & 0 & 2 \\ 
0 & 0 & 0 
\end{array} \right).
$
Then, Ind($A$)$=2$, 
\begin{equation*}
A^{\dagger} = \left( \begin{array}{ccc}
\frac{1}{2} & -\frac{1}{4} & 0 \\ 
0 & 0 & 0 \\
0 & \frac{1}{2} & 0 
\end{array} \right), \quad
A^{d} = \left( \begin{array}{ccc}
\frac{1}{2} & 0 & \frac{1}{4} \\ 
0 & 0 & 0 \\ 
0 & 0 & 0 
\end{array} \right),
\end{equation*}
\begin{equation*}
A^{d,\dagger} = \left( \begin{array}{ccc}
\frac{1}{2} & 0 & 0 \\ 
0 & 0 & 0 \\ 
0 & 0 & 0 
\end{array} \right), \quad
A^{\dagger ,d} =\left( \begin{array}{ccc}
\frac{1}{2} & 0 & \frac{1}{4} \\ 
0 & 0 & 0 \\ 
0 & 0 & 0 
\end{array} \right).
\end{equation*}
It is evident that $A^{\dagger , d}= A^{d} $ and $ A^{d,\dagger} \neq A^{d} . $ This fact is due to the core-EPness of matrix $A$ fails.
\end{example}
Employing  a similar method as in the proof of Theorem \ref{as20},  the following holds
\begin{corollary}\label{TTas20}
Let $ A \in M_{n}(\mathbb{C})$ be 
a core-EP matrix. The  following conditions are equivalent: 
\begin{enumerate}
\item the MPDMP  matrix associated with $A$ is equal to the  DMP  inverse of $A$,
\item the MPDMP matrix associated with $A$ is equal to the MPD  inverse of $A$. 
\end{enumerate}
\end{corollary}
We have proved that in the class of core-EP matrices,  DMP, MPD inverses, and MPDMP matrix associated with $A$,  are the same.
The hypothesis of core-EPness in Corollary \ref{TTas20} is essential, which means that the class of core-EP matrices is the biggest one on which those three matrices coincide.
\begin{example}\label{EXAMV5}
Let
$
A = \left( \begin{array}{ccc}
1 & 0 & 0 \\ 
1 & 0 & 1 \\ 
0 & 0 & 0 
\end{array} \right).
$
Then, Ind($A$)$=2$, 
\begin{equation*}
A^{\dagger} = \left( \begin{array}{ccc}
1 & 0 & 0 \\ 
0 & 0 & 0 \\
-1 & 1 & 0 
\end{array} \right), \quad
A^{d} = \left( \begin{array}{ccc}
1 & 0 & 0 \\ 
1 & 0 & 0 \\ 
0 & 0 & 0 
\end{array} \right), \quad
A^{\dagger ,d,\dagger} = \left( \begin{array}{ccc}
1 & 0 & 0 \\ 
0 & 0 & 0 \\ 
0 & 0 & 0 
\end{array} \right),
\end{equation*}
\begin{equation*}
A^{d,\dagger} = \left( \begin{array}{ccc}
1 & 0 & 0 \\ 
1 & 0 & 0 \\ 
0 & 0 & 0 
\end{array} \right), \quad
A^{\dagger ,d} =\left( \begin{array}{ccc}
1 & 0 & 0 \\ 
0 & 0 & 0 \\ 
0 & 0 & 0 
\end{array} \right).
\end{equation*}
It is evident that $A^{\dagger , d,\dagger}= A^{\dagger ,d} $ and $ A^{\dagger ,d,\dagger} \neq A^{d,\dagger} . $ This fact is due to the core-EPness of matrix $A$ fails.
\end{example}
\begin{proposition}\cite[Theorem 2.5 and Remark 2.9]{Malik}
Let $ A \in M_{n}(\mathbb{C})$ be of the form (\ref{h}). Then
\begin{equation}\label{mdhkp}
\begin{small}
A^{d,\dagger} = U \left( \begin{array}{cc}
(\Sigma Q )^{d} & 0 \\ 
0 & 0
\end{array} \right) U^{*}
\end{small}
\quad \&\quad 
\begin{small}
A^{\dagger ,d} = U \left( \begin{array}{cc}
Q^{*}\hat{Q} & Q^{*}\hat{Q}(\Sigma Q)^{d}\Sigma P \\ 
P^{*}\hat{Q} & P^{*}\hat{Q}(\Sigma Q)^{d}\Sigma P 
\end{array} \right) U^{*}.
\end{small}
\end{equation}
\end{proposition}
\begin{theorem}
Let $ A \in M_{n}(\mathbb{C})$ be 
a core-EP matrix. Then
\begin{enumerate}[label=(\roman*)]
\item 
$ A^{\dagger ,d}A=AA^{\dagger ,d},$
\item
$ A^{\dagger ,d}A^{d}=A^{d}A^{\dagger ,d},$
\item
$ A^{\dagger ,d}A_{c}=A_{c}A^{\dagger ,d},$
\item
$ A^{\dagger,d,\dagger}A^{\dagger , d}=A^{\dagger , d}A^{\dagger,d,\dagger}, $
\item
$A_{c}=A^{c,\dagger}A^{2}=A^{\dagger ,d}A^{2}=A^{d,\dagger}A^{2}$,
\item
$ Q_{A}A_{c}=A_{c}Q_{A}=A_c $. 
\end{enumerate}
\end{theorem}
\begin{proof}
Let $ A \in M_{n}(\mathbb{C})$ be written as in (\ref{h}).
By \cite[Lemma 3.2]{am16}, we get 
\begin{equation}\label{40}
Q^{*}\hat{Q} = \left( \Sigma Q \right)^{d}\quad \& \quad 
P^{*}\hat{Q} = 0\quad \& \quad 
( \Sigma Q )^{d} \Sigma P = 0. 
\end{equation}

$ (i) $ 
By using  (\ref{h}),  (\ref{KL}) and (\ref{mdhkp}), we have
\begin{equation}\label{151}
\begin{split}
A^{\dagger ,d}A&=U \left( \begin{array}{cc}
Q^{*}\hat{Q}\Sigma Q &Q^{*}\hat{Q}\Sigma P \\ 
P^{*}\hat{Q}\Sigma Q & P^{*}\hat{Q}\Sigma P
\end{array} \right) U^{*}, \\ 
AA^{\dagger ,d}&=U \left( \begin{array}{cc}
\Sigma \hat{Q} &(\Sigma Q)^{d}\Sigma P \\ 
0 & 0
\end{array} \right) U^{*}. 
\end{split}
\end{equation}
Now, from (\ref{40}) and (\ref{151}), we obtain $ A^{\dagger ,d}A=AA^{\dagger ,d}.$

$ (ii) $ It follows from $(i)$ and by using that $A^d$ is a polynomial on $A$ (see \cite{Ben1}).

$ (iii) $ By using  (\ref{cm}) and  (\ref{mdhkp}), we have
\begin{equation}\label{a69}
\begin{split}
A^{\dagger ,d}A_{c}&=U \left( \begin{array}{cc}
Q^{*}\hat{Q}\Sigma Q &Q^{*}\hat{Q}\Sigma P \\ 
P^{*}\hat{Q}\Sigma Q & P^{*}\hat{Q}\Sigma P
\end{array} \right) U^{*},\\ 
A_{c}A^{\dagger ,d}&=U 
\left( \begin{array}{cc}
\Sigma \hat{Q} &(\Sigma Q)^{d}\Sigma P \\ 
0 & 0
\end{array} \right) U^{*}.
\end{split}
\end{equation}
Therefore, by (\ref{40}) and (\ref{a69}), we have $ A^{\dagger ,d}A_{c}=A_{c}A^{\dagger ,d}.$

$ (iv) $ Let $ A\in M_{n}(\mathbb{C}) $ be as in (\ref{h}) and denote $\tilde{\Sigma}=\hat{Q}((\Sigma Q)^{d})^{2}$.  
By using Definition \ref{ASER}, (\ref{KL}) and (\ref{md}), we have that
\begin{equation}\label{WMM2}
A^{\dagger ,d,\dagger}=U \left( \begin{array}{cc}
Q^{*}\tilde{\Sigma} &0\\ 
P^{*}\tilde{\Sigma} & 0
\end{array} \right) U^{*}
\end{equation}
By using (\ref{mdhkp}) and (\ref{WMM2}), we have
\begin{equation}\label{ab270}
\begin{split}
A^{\dagger ,d,\dagger}A^{\dagger , d}&= U \left( \begin{array}{cc}
Q^{*}\tilde{\Sigma}Q^{*}\hat{Q} & Q^{*}\tilde{\Sigma}Q^{*}\hat{Q}(\Sigma Q)^{d}\Sigma P \\ 
P^{*}\tilde{\Sigma}Q^{*}\hat{Q} & P^{*}\tilde{\Sigma}Q^{*}\hat{Q}(\Sigma Q)^{d}\Sigma P
\end{array} \right) U^{*}, \\
A^{\dagger , d}A^{\dagger,d,\dagger}&= U \left( \begin{array}{cc}
Q^{*}\tilde{\Sigma}(\Sigma Q)^{d} & 0 \\ 
P^{*}\tilde{\Sigma}(\Sigma Q)^{d} & 0
\end{array} \right) U^{*}. 
\end{split}
\end{equation}
Therefore, by (\ref{40}) and (\ref{ab270}), we have that
$ A^{\dagger,d,\dagger}A^{\dagger , d}=A^{\dagger , d}A^{\dagger,d,\dagger}.$

$ (v) $ 
From \cite[Theorem 3.3]{am16} we have that $A^{c,\dag}=A^d$ provided that $A$ is a core-EP matrix. So, $A_c=AA^dA=A^d A^2=A^{c,\dag} A^2$. Similarly, by \cite[Theorem 3.6]{am16}, $A^{d,\dag}=A^{\dag,d}$ holds, from which it only remais to prove that $A_c=A^{d,\dag}A^2$. In fact, 
\begin{align}\label{a270} 
A^{d,\dagger}A^{2}&= U \left( \begin{array}{cc}
\Sigma \hat{Q}(\Sigma Q) & \Sigma \hat{Q} \Sigma P \\ 
0 & 0
\end{array} \right) U^{*}. 
\end{align}
Therefore, by (\ref{40}) and (\ref{a270}), we have 
$A_{c}=A^{d,\dagger}A^{2}. $

$ (vi) $ By definition and using that $A$ is a core-EP matrix we have that \\
$Q_A A_c =A^\dag A A A^d A=A^\dag (A A^d A) A=(A A^d A) A^\dag A=A A^d A=A_c$. Similarly,\\ $A_c Q_A = A A^d A A^\dag A=A A^d A=A_c$. Hence, we have the required equalities.
\end{proof}
\begin{theorem}
Let $ A \in M_{n}(\mathbb{C})$ be 
a core-EP matrix with Ind$(A)=k.$ Then $ A^{\dagger ,d,\dagger}$ is the unique matrix $X$ that satisfies
\begin{equation}\label{kj43}
A^{3}X=P_{\mathcal{R}(A^k) , \mathcal{N}(A^k)}, \quad \quad \mathcal{R}(X) \subseteq \mathcal{R}(A^{k}).
\end{equation}
\end{theorem}
\begin{proof}
We have that
\begin{align*}
\mathcal{R}(A^{3}A^{\dagger ,d,\dagger})&=\mathcal{R}(A^{d}A^{2}A^{\dagger})\subseteq \mathcal{R}(A^{d})=\mathcal{R}(A^{k})=
\mathcal{R}(A^{d}A^{2}A^{\dagger}A^{k}) \subseteq \mathcal{R}(A^{d}A^{2}A^{\dagger})\\
\mathcal{N}(A^{3}A^{\dagger ,d,\dagger})&=\mathcal{N}(A_{c}A^{\dagger})=\mathcal{N}(A^{\dagger}A_{c})\subseteq \mathcal{N}(A^{d}A^{\dagger}A_{c})=
\mathcal{N}(A^{d})=\mathcal{N}(A^{k}) 
\\&\subseteq \mathcal{N}((A^{d})^{k}A^{k}) =\mathcal{N}(A^{d}A) \subseteq \mathcal{N}(A^{\dagger}A_{c}).
\end{align*}
By \cite{am16},  $ A $ is core-EP matrix, we have 
\begin{equation*}
\mathcal{R}(A^{\dagger ,d,\dagger})=\mathcal{R}(A^{\dagger}A^{ d}A^{\dagger})=\mathcal{R}(A^{\dagger}A^{k}(A^{ d})^{k}A^{d}A^{\dagger})
=\mathcal{R}(A^{k}A^{\dagger}(A^{ d})^{k}A^{d}A^{\dagger})\subseteq \mathcal{R}(A^{k}).
\end{equation*} 
Suppose that $Y_1, Y_2$ satisfy \eqref{kj43}. Then $A^{3}Y_1=A^{3}Y_2=P_{\mathcal{R}(A^k) , \mathcal{N}(A^k)}$, 
$\mathcal{R}(Y_1) \subseteq \mathcal{R}(A^{k})$ and $\mathcal{R}(Y_2) \subseteq \mathcal{R}(A^{k})$. Since $A^{3}(Y_1-Y_2)=0$, we get 
$\mathcal{R}(Y_1-Y_2) \subseteq \mathcal{N}(A^3)$. From $\mathcal{R}(Y_1) \subseteq \mathcal{R}(A^{k})$ and $\mathcal{R}(Y_2) \subseteq \mathcal{R}(A^{k})$ we get 
$\mathcal{R}(Y_1-Y_2) \subseteq \mathcal{R}(A^{k})$, that is\\ $\mathcal{R}(Y_1-Y_2) \subseteq \mathcal{R}(A^{k})\cap \mathcal{N}(A^3) \subseteq 
 \mathcal{R}(A^{k})\cap \mathcal{N}(A^k)=\{0\}.$ Thus, $Y_1=Y_2$.
\end{proof}
Now, we are looking for necessary and sufficient conditions for a matrix to be a core-EP matrix.
\begin{theorem}\label{ehsan100}
Let $ A\in M_{n}(\mathbb{C})$. Then $ A$ is core-EP matrix 
if and only if  \\
$ A^{\dagger ,d,\dagger}=(A^{d})^{3}$. 
\end{theorem}
\begin{proof}
By (\ref{md}), we have that
\begin{equation}\label{ehsan78}
(A^{d})^{3} = U \left( \begin{array}{cc}
\left((\Sigma Q)^{d}\right)^{3} & \left((\Sigma Q)^{d}\right)^{4}\Sigma P \\ 
0 & 0
\end{array} \right) U^{*}. 
\end{equation}
By (\ref{WMM2}) and (\ref{ehsan78}), the equality $ A^{\dagger ,d,\dagger}=(A^{d})^{3}$ if and only if the  following conditions hold:
\begin{align}
Q^{*}\tilde{\Sigma}&= \left((\Sigma Q)^{d}\right)^{3},\label{30k}
\\\left((\Sigma Q)^{d}\right)^{4}\Sigma P  &= 0 ,\label{31k}
\\P^{*}\tilde{\Sigma}&= 0,\label{32k}   
\end{align}
By \cite[Lemma 3.2]{am16}, $A$ is a core-EP matrix
if and only if the following conditions hold: 
\begin{equation*}
(a) Q^{*}\hat{Q} = \left( \Sigma Q \right)^{d},\\\quad
(b) P^{*}\hat{Q} = 0, \\ \quad 
(c) \left( \Sigma Q \right)^{d} \Sigma P = 0. 
\end{equation*}
By right-multiplying the equations (\ref{30k}) and (\ref{32k}) 
by $ (\Sigma Q)^{2} $, 
we get that the equations (\ref{30k}) and (\ref{32k}) are equivalent to the equations (a) and (b). Pre-multiplying the equalities (\ref{31k}) by $ (\Sigma Q)^{3} $,
we get $ ( \Sigma Q )^{d}  \Sigma P = 0, $ which gives (c)  and the result hold. 
\end{proof}
Employing a similar method as in the proof of Theorem \ref{ehsan100}, \eqref{h},\eqref{KL}, \eqref{cm} and  (\ref{mdhkp}),  the following holds.
\begin{corollary}\label{JKW5}
Suppose that  $ A\in M_{n}(\mathbb{C}) $. Then $A$ is a core-EP matrix 
if and only if $ A^{ \dagger ,d,\dagger}A^{d ,\dagger}=(A^{d})^{4}$ if and only if  
$ A^{\dagger ,d,\dagger}A=AA^{\dagger ,d,\dagger}$ if and only if  $ A^{\dagger ,d,\dagger}A_c=A_cA^{\dagger ,d,\dagger}$.
\end{corollary}
\begin{theorem}\label{ehsan36}
Suppose that $ A\in M_{n}(\mathbb{C})$. Then $A$ is a core-EP matrix  if and only if 
$A^{\dagger ,d,\dagger}A^{d}=A^{d}A^{\dagger ,d,\dagger}$.
\end{theorem}
\begin{proof}
By using (\ref{md}) and (\ref{WMM2}), we get
\begin{align}\label{297}
A^{\dagger ,d,\dagger}A^{d}&=U \left( \begin{array}{cc}
Q^{*}\tilde{\Sigma}(\Sigma Q)^{d} &Q^{*}\tilde{\Sigma}((\Sigma Q)^{d})^{2}\Sigma P\\ 
P^{*}\tilde{\Sigma}(\Sigma Q)^{d} & P^{*}\tilde{\Sigma}((\Sigma Q)^{d})^{2}\Sigma P
\end{array} \right) U^{*}, \\
A^{d}A^{\dagger ,d,\dagger}&=U \left( \begin{array}{cc}
((\Sigma Q)^{d})^{4} & 0 \\ 
0 & 0
\end{array} \right) U^{*}. \label{298}
\end{align}
By (\ref{297}) and (\ref{298}), the equality $ A^{\dagger ,d,\dagger}A^{d}=A^{d}A^{\dagger ,d,\dagger}$ holds if and only if
the following conditions fulfill: 
\begin{align}
Q^{*}\tilde{\Sigma}(\Sigma Q)^{d} &= ((\Sigma Q)^{d})^{4},\label{299}
\\Q^{*}\tilde{\Sigma}((\Sigma Q)^{d})^{2}\Sigma P &= 0,\label{300a} 
\\P^{*}\tilde{\Sigma}(\Sigma Q)^{d} &= 0 ,\label{301a}
\\P^{*}\tilde{\Sigma}((\Sigma Q)^{d})^{2}\Sigma P &= 0 ,\label{301aa}
\end{align}
By \cite[Lemma 3.2]{am16}, $A$ is a core-EP matrix
if and only if the following conditions hold: 
\begin{equation*}\label{a20} 
(a) Q^{*}\hat{Q} = \left( \Sigma Q \right)^{d},\\\quad
(b) P^{*}\hat{Q} = 0, \\ \quad 
(c) \left( \Sigma Q \right)^{d} \Sigma P = 0. 
\end{equation*}
Right-multiplying the equation (\ref{299}) and (\ref{301a}) by $(\Sigma Q)^{3} $, respectively, we arrive at 
(a) and (b).
Since $ QQ^{*}+PP^{*}=I_{r},$ we can pre-multiply equations (\ref{300a}) and (\ref{301aa}) by $ (\Sigma Q)^{4} $ and $ (\Sigma Q)^{3}\Sigma P $, respectively, This results in 
$ (\Sigma Q)^{d}\Sigma P=0,$ that is, (c) is equivalent to (\ref{300a}) and (\ref{301aa}) and the result hold.
\end{proof}
Employing  a similar method as in the proof of Theorem \ref{ehsan36} and  (\ref{mdhkp}),  the following holds.
\begin{corollary}\label{JKW59}
Suppose that  $ A\in M_{n}(\mathbb{C}) $. Then $A$ is a core-EP matrix 
 if and only if $ A^{\dagger ,d,\dagger}A^{d}=(A^{d ,\dagger})^{4}.$
\end{corollary}
\begin{corollary}\label{KheiEHSA7889}
Suppose that  $ A \in M_{n}(\mathbb{C}) $. Then 
the following are equivalent: 
\begin{enumerate}[label=(\roman*)]
\item
$ A$ is a core-EP matrix,                   
\item
$A^{\dagger ,d,\dagger}=(A^{d})^{3}, $ 
$\hfill$ by Theorem \ref{ehsan100},
\item
$ A^{ \dagger ,d,\dagger}A^{d ,\dagger}=(A^{d})^{4}$,   
$\hfill$ by Corollary \ref{JKW5},
\item
$A^{\dagger ,d,\dagger}A=AA^{\dagger ,d,\dagger}$, 
$\hfill$ by Corollary \ref{JKW5},
\item
$A^{\dagger ,d,\dagger}A_{c}=A_{c}A^{\dagger ,d,\dagger}, $ 
$\hfill$ by Corollary \ref{JKW5},
\item
$A^{\dagger ,d,\dagger}A^{d}=A^{d}A^{\dagger ,d,\dagger}$, 
$\hfill$ by Theorem \ref{ehsan36},
\item
$ A^{\dagger ,d,\dagger}A^{d}=(A^{d ,\dagger})^{4}$, 
$\hfill$ by Corollary \ref{JKW59}.
\end{enumerate}
\end{corollary}
In what follows, we are looking for equivalent conditions such that $ A^{c,\dagger } $ is an EP matrix.
\begin{lemma}\label{DD}
Let $ A \in M_{n}(\mathbb{C}) $ be as in (\ref{h}). If 
$\Delta P=0$ 
then $P^*\Delta Q=0$, where \\
$\Delta=\hat{Q}(\hat{Q})^{\dagger}$.
\end{lemma}
\begin{proof}
It is clear that $ \Delta = \hat{Q}(\hat{Q})^{\dagger} $ is an orthogonal projector. Thus, 
$ \Delta $ is hermitian. If 
$ \Delta P=0, $ then $ P^{*}\Delta =0 $ and this implies 
$ P^*\Delta Q=0. $
\end{proof}
\begin{theorem}\cite[Theorem 2.10]{new}\label{ADC} 
Let $ A \in M_{n}(\mathbb{C}) $ be as in (\ref{h}) and $ \Delta = \hat{Q}(\hat{Q})^{\dagger}$. Then $ A^{c,\dagger } $ is an EP matrix  if and only if 
\begin{enumerate}[label=(\roman*)]
\item
$Q^{*}\Delta Q=(\hat{Q})^{\dagger}\hat{Q}$,
\item
$\Delta P=0$,
\item
$P^*\Delta Q=0.$
\end{enumerate}
\end{theorem}
We can improve the previous Theorem \ref{ADC}.
In the above theorem the authors show that $ A^{c,\dagger } $ is an EP matrix  if and only if three conditions hold. But Lemma \ref{DD} shows that $(ii)$ implies $(iii)$. Therefore the condition $(iii)$ in \cite[Theorem 2.10]{new} is redundant.

\begin{corollary}\label{WQRT}
Let $ A \in M_{n}(\mathbb{C})$. Then $ A^{c,\dagger } $ is an EP matrix  if and only if \\
$ A^{\dagger ,d}(A^{c,\dagger })^{\dagger}=(A^{c,\dagger })^{\dagger}A^{d ,\dagger}. $
\end{corollary}
\begin{proof}
Suppose that $ A \in M_{n}(\mathbb{C})$ be as in (\ref{h}). By using (\ref{mdhkp}) 
and \cite[ the proof of Theorem 2.6 (1)]{am16},
we get 
\begin{align*}
A^{\dagger ,d}(A^{c,\dagger })^{\dagger}&=U \left( \begin{array}{cc}
Q^{*}\hat{Q}(\hat{Q})^{\dagger}Q &Q^{*}\hat{Q}(\hat{Q})^{\dagger} P\\ 
P^{*}\hat{Q}(\hat{Q})^{\dagger}Q & P^{*}\hat{Q}(\hat{Q})^{\dagger}P
\end{array} \right) U^{*}, 
\end{align*}
which is hermitian, and 
\begin{align*}
(A^{c,\dagger })^{\dagger}A^{d ,\dagger}&=U \left( \begin{array}{cc}
(\hat{Q})^{\dagger}\hat{Q} &0\\ 
0 & 0
\end{array} \right) U^{*}. 
\end{align*}
Then $ A^{\dagger ,d}(A^{c,\dagger })^{\dagger}=(A^{c,\dagger })^{\dagger}A^{d, \dagger}$ if and only if the   folowing conditions hold:
\begin{align}
Q^{*}\hat{Q}(\hat{Q})^{\dagger}Q&=(\hat{Q})^{\dagger}\hat{Q},\label{162}
\\Q^{*}\hat{Q}(\hat{Q})^{\dagger} P&=0,\label{163}
\\P^{*}\hat{Q}(\hat{Q})^{\dagger}P&=0.\label{165} 
\end{align}
Thus, by Theorem \ref{ADC}, we know that $ A^{c,\dagger } $ is EP matrix if and only if 
\begin{align}
Q^{*}\hat{Q}(\hat{Q})^{\dagger} Q &=(\hat{Q})^{\dagger}\hat{Q},\label{166}
\\\hat{Q}(\hat{Q})^{\dagger}P&=0.\label{167}
\end{align}
The equations in (\ref{162}) and (\ref{166}) are the same. By pre-multiplying (\ref{163}) by $Q $ and (\ref{165}) by $P $
and utilizing (\ref{KL}), we arrive at
$ \hat{Q}(\hat{Q})^{\dagger}P = 0, $ which is (\ref{167}).
\end{proof}
We obtain properties by using  the MPDMP matrix associated with $A$.
\begin{theorem}\label{ER}
Let $ A \in M_{n}(\mathbb{C}) $ be written as in (\ref{h}). Then
\begin{equation}\label{ehsan80}
(A^{\dagger ,d ,\dagger})^{\dagger}= 
U \left( \begin{array}{cc}
(\tilde{\Sigma})^{\dagger}Q &(\tilde{\Sigma})^{\dagger}P\\ 
0 & 0
\end{array} \right) U^{*}.
\end{equation}
\end{theorem}
\begin{proof} 
Assume that $A $ is represented  as 
in (\ref{h}) and 
\begin{equation*} 
X=U \left( \begin{array}{cc}
(\tilde{\Sigma})^{\dagger}Q &(\tilde{\Sigma})^{\dagger}P\\ 
0 & 0
\end{array} \right) U^{*}.
\end{equation*}
By using (\ref{KL}) and \eqref{WMM2}, we have that
\begin{align*} 
A^{\dagger ,d,\dagger}XA^{\dagger ,d,\dagger}
&=U \left( \begin{array}{cc}
Q^{*}\tilde{\Sigma} &0\\ 
P^{*}\tilde{\Sigma} & 0
\end{array} \right) U^{*}=A^{\dagger ,d,\dagger}, \\ 
XA^{\dagger ,d,\dagger}X
&=U \left( \begin{array}{cc}
(\tilde{\Sigma})^{\dagger}Q &(\tilde{\Sigma})^{\dagger}P\\ 
0 & 0
\end{array} \right) U^{*}=X, \\
(A^{\dagger ,d,\dagger}X)^{*}
&=U\left( \begin{array}{cc}
Q^{*}\tilde{\Sigma}(\tilde{\Sigma})^{\dagger}Q & Q^{*}\tilde{\Sigma}(\tilde{\Sigma})^{\dagger}P \\ 
P^{*}\tilde{\Sigma}(\tilde{\Sigma})^{\dagger}Q & P^{*}\tilde{\Sigma}(\tilde{\Sigma})^{\dagger}P
\end{array} \right)U^{*} =A^{\dagger ,d,\dagger}X, \\ 
(XA^{\dagger ,d,\dagger})^{*}
&=U \left(\begin{array}{cc}
(\tilde{\Sigma})^{\dagger}(\tilde{\Sigma})^{*} & 0 \\ 
0 & 0
\end{array} \right) U^{*}
=XA^{\dagger ,d,\dagger}.
\end{align*}
The matrix $X$ satisfies four equations \eqref{MMoo}. Suppose both $X_1$ and $X_2$ also satisfy four equations each. In order to establish uniqueness, we proceed as follows 
\begin{align*} 
X_1&=X_1(AX_1)^{*}=X_1X_1^{*}A^{*}=X_1X_1^{*}(AX_1A)^{*}=
X_1X_1^{*}A^{*}Z^{*}A^{*}
=X_1(AX_1)^{*}(AX_2)^{*}
\\&=X_1AX_2
=X_1AX_2AX_2=(X_1A)^{*}(X_2A)^{*}X_2
=A^{*}X_1^{*}A^{*}X_2^{*}X_2=(X_2A)^{*}X_2=X_2.
\end{align*}
\end{proof}
We obtain three equivalent conditions for $A^{\dagger ,d,\dagger}, A^{C,\tddd}$ and $  A^{c,\dagger}$ to be an EP matrix.
\begin{theorem}\label{ehsank1}
Assume that $A $ is represented  as 
in (\ref{h}). Then $ A^{\dagger ,d,\dagger } $ is an EP matrix if and only if 
\begin{enumerate}[label=(\roman*)]
\item
$Q^{*}\hat{\Delta} Q=(\tilde{\Sigma})^{\dagger}\tilde{\Sigma}$,
\item
$\hat{\Delta} P=0$,
\end{enumerate}
where $ \hat{\Delta} =\tilde{\Sigma}( \tilde{\Sigma})^{\dagger}$.
\end{theorem}
\begin{proof}
By (\ref{KL}), \eqref{WMM2} and \eqref{ehsan80}, we have that 
\begin{align*}
A^{\dagger ,d,\dagger}(A^{\dagger ,d,\dagger })^{\dagger}&=U \left( \begin{array}{cc}
Q^{*}\tilde{\Sigma}(\tilde{\Sigma})^{\dagger}Q &Q^{*}\tilde{\Sigma}(\tilde{\Sigma})^{\dagger}P\\ 
P^{*}\tilde{\Sigma}(\tilde{\Sigma})^{\dagger}Q & P^{*}\tilde{\Sigma}(\tilde{\Sigma})^{\dagger}P
\end{array} \right) U^{*}, \\
(A^{\dagger ,d,\dagger})^{\dagger}A^{\dagger ,d,\dagger}&=U \left( \begin{array}{cc}
( \tilde{\Sigma})^{\dagger}\tilde{\Sigma} &0\\ 
0 & 0
\end{array} \right) U^{*}. 
\end{align*}
Then  $ A^{\dagger ,d,\dagger}(A^{\dagger ,d,\dagger })^{\dagger}=(A^{\dagger ,d,\dagger})^{\dagger}A^{\dagger ,d,\dagger} $  if and only if  the below conditions hold:
\begin{align}
Q^{*}\tilde{\Sigma}( \tilde{\Sigma})^{\dagger}Q&=( \tilde{\Sigma})^{\dagger}\tilde{\Sigma},\label{192t}
\\Q^{*}\tilde{\Sigma}(\tilde{\Sigma})^{\dagger}P&=0,\label{193t}
\\ P^{*}\tilde{\Sigma}( \tilde{\Sigma})^{\dagger}Q &= 0,\label{194t} 
\\ P^{*}\tilde{\Sigma}( \tilde{\Sigma})^{\dagger}P&=0.\label{195t} 
\end{align}
Observe that the equation (\ref{192t}) is equivalent to Theorem \ref{ADC}$(i)$.
 Using (\ref{KL}), by  left-multiplying the equations (\ref{193t}) and (\ref{195t}) by $Q $ and $ P$,
respectively, we obtain \\
$\tilde{\Sigma}(\tilde{\Sigma})^{\dagger}P=0$, equivalent to  Theorem \ref{ADC}$(ii)$.
\end{proof}
\begin{corollary}\label{XSDF11}
Let $A \in M_{n}(\mathbb{C})$ be written as in (\ref{h}). If 
$  A^{\dagger ,d,\dagger }$ is an EP matrix, then
\begin{enumerate}
\item
$ [PP^{*}, \hat{\Delta}]=0, $
\item
$ [QQ^{*}, \hat{\Delta}]=0, $
\item
$\hat{\Delta} =Q( \tilde{\Sigma})^{\dagger}\tilde{\Sigma}Q^{*}$,
\end{enumerate}
where $ [A, B]=AB-BA.$
\end{corollary}
\begin{proof}
Suppose that  $ A^{\dagger ,d,\dagger}(A^{\dagger ,d,\dagger })^{\dagger}=(A^{\dagger ,d,\dagger})^{\dagger}A^{\dagger ,d,\dagger} $. Then 
(\ref{193t}) and (\ref{194t}) hold.
Pre and post multiplying (\ref{193t}) by $ Q $ and $ P^{*}, $ respectively, and moreover, pre and post multiplying (\ref{194t}) by 
$ P $ and 
$ Q^{*} $, respectively, we have 
\begin{equation}\label{ase234}
\begin{split}
QQ^{*}\tilde{\Sigma}( \tilde{\Sigma})^{\dagger}PP^{*}&= 0\\ 
PP^{*}\tilde{\Sigma}(\tilde{\Sigma})^{\dagger}QQ^{*}&= 0.
\end{split}
\end{equation}
Using (\ref{KL}) and (\ref{ase234}), we get
\begin{equation}\label{ase236}
\begin{split}
(I_{r}-PP^{*})\tilde{\Sigma}( \tilde{\Sigma})^{\dagger}PP^{*}&= 0\\ 
(I_{r}-QQ^{*})\tilde{\Sigma}(\tilde{\Sigma})^{\dagger}QQ^{*}&= 0.
\end{split}
\end{equation}
Then (\ref{ase236}) can be written as 
\begin{equation}\label{ase238}
\begin{split}
\tilde{\Sigma}(\tilde{\Sigma})^{\dagger}PP^{*}&= PP^{*}\tilde{\Sigma}(\tilde{\Sigma})^{\dagger}PP^{*}, \\
\tilde{\Sigma}( \tilde{\Sigma})^{\dagger}QQ^{*}&=QQ^{*}\tilde{\Sigma}( \tilde{\Sigma})^{\dagger}QQ^{*}.
\end{split}
\end{equation}
Using (\ref{KL}), (\ref{ase234}) and (\ref{ase238}), we obtain
\begin{equation*}
\begin{split}
&\tilde{\Sigma}(\tilde{\Sigma})^{\dagger}PP^{*}= PP^{*}\tilde{\Sigma}(\tilde{\Sigma})^{\dagger}PP^{*}\\
&= PP^{*}\tilde{\Sigma}(\tilde{\Sigma})^{\dagger}PP^{*}
+PP^{*}\tilde{\Sigma}(\tilde{\Sigma})^{\dagger}QQ^{*}
\\&=PP^{*}\tilde{\Sigma}( \tilde{\Sigma})^{\dagger}(PP^{*}+QQ^{*})
=PP^{*}\tilde{\Sigma}(\tilde{\Sigma})^{\dagger}\\
&\tilde{\Sigma}( \tilde{\Sigma})^{\dagger}QQ^{*}
=QQ^{*}\tilde{\Sigma}( \tilde{\Sigma})^{\dagger}QQ^{*}
+QQ^{*}\tilde{\Sigma}(\tilde{\Sigma})^{\dagger}PP^{*}
\\&=QQ^{*}\tilde{\Sigma}(\tilde{\Sigma})^{\dagger}(QQ^{*}+PP^{*})
=QQ^{*}\tilde{\Sigma}(\tilde{\Sigma})^{\dagger}.
\end{split}
\end{equation*}
Therefore, $ [PP^{*} , \hat{\Delta}]=0 $ and 
$ [QQ^{*} ,\hat{\Delta}]=0$.

3. By Theorem \ref{ADC}$(i)$,
premultiplying
$Q^{*}\hat{\Delta} Q=(\tilde{\Sigma})^{\dagger}\tilde{\Sigma}$ by $Q$, Moreover, $ \hat{\Delta} = \tilde{\Sigma}( \tilde{\Sigma})^{\dagger}$ is an orthogonal projector. Thus, 
$ \hat{\Delta} $ is hermitian. By Theorem \ref{ADC}$(ii)$,
$ \hat{\Delta}P=0, $ then $ P^{*}\hat{\Delta} =0 $ and this implies 
$P^*\hat{\Delta} Q=0,$
by $P$ and adding them  and using (\ref{KL}), we get $\hat{\Delta} Q=Q( \tilde{\Sigma})^{\dagger}\tilde{\Sigma} $.
Now, post-multyplying $\hat{\Delta} Q=Q( \tilde{\Sigma})^{\dagger}\tilde{\Sigma} $ by $Q^*$ and $\hat{\Delta} P=0$ by $P^*$ and adding then, we get,
$\hat{\Delta}  =Q(\tilde{\Sigma})^{\dagger}\tilde{\Sigma} Q^*$.
\end{proof}
Now, we consider the CCE-inverse $A^{C,\tddd}=A^\dag A A^{\tddd} A A^\dag$ of $A \in M_{n}(\mathbb{C})$ defined in \cite{Zuo}.

Employing  a similar method as in the proof of Theorem \ref{ER},  the following hold.
\begin{corollary}\label{ERUPO}
Let $ A \in M_{n}(\mathbb{C}) $ be written as in (\ref{h}). Then
\begin{equation*}
(A^{C,\tddd})^{\dagger}= 
U \left( \begin{array}{cc}
(\tilde{Q})^{\dagger}Q &(\tilde{Q})^{\dagger}P\\ 
0 & 0
\end{array} \right) U^{*}.
\end{equation*}
where $ \tilde{Q}=Q(\Sigma Q)^{\tddd}.$
\end{corollary}
Employing  a similar method as in the proofs of Theorem \ref{ehsank1}, Corollarys \ref{XSDF11}, \ref{ERUPO} and \cite[Theorem 3.2]{Zuo},  the following hold.
\begin{corollary}
Assume that $A $ is represented  as 
in (\ref{h}). Then $ A^{C,\tddd} $ is an EP matrix if and only if 
\begin{enumerate}[label=(\roman*)]
\item
$Q^{*}\tilde{\Delta} Q=(\tilde{Q})^{\dagger}\tilde{Q}$,
\item
$\tilde{\Delta} P=0$,
\end{enumerate}
 Moreover, If $  A^{C,\tddd}$ is an EP matrix, then
\begin{enumerate}
\item
$ [PP^{*}, \tilde{\Delta}]=0, $
\item
$ [QQ^{*}, \tilde{\Delta}]=0, $
\item
$\tilde{\Delta} =Q( \tilde{Q})^{\dagger}\tilde{Q}Q^{*}$,
\end{enumerate}
where $ \tilde{\Delta} =\tilde{Q}( \tilde{Q})^{\dagger}$.
\end{corollary}
\begin{theorem} 
Let $ A \in M_{n}(\mathbb{C}) $ be written as in (\ref{h}). If $ (\Sigma Q)^{\tddd}=(\Sigma Q)^{d} $ , then 
 $ A^{c,\dagger } $ is an EP matrix if and only if 
$ A^{C,\tddd}(A^{c,\dagger })^{\dagger}=(A^{c,\dagger })^{\dagger} A^{C,\tddd}$.
\end{theorem}
\begin{proof}
Assume that $A $ is represented  as 
in (\ref{h}). By using  the proof of  \cite[Theorem 2.6 (1)]{am16} and
\cite[Theorem 3.2]{Zuo}, we have 
\begin{align*}
A^{C, \tddd}(A^{c,\dagger })^{\dagger}&=U \left( \begin{array}{cc}
Q^{*}\tilde{Q}(\hat{Q})^{\dagger}Q &Q^{*}\tilde{Q}(\hat{Q})^{\dagger} P\\ 
P^{*}\tilde{Q}(\hat{Q})^{\dagger}Q & P^{*}\tilde{Q}(\hat{Q})^{\dagger}P
\end{array} \right) U^{*}, \\
(A^{c,\dagger })^{\dagger}A^{C ,\tddd}&=U \left( \begin{array}{cc}
(\hat{Q})^{\dagger}\tilde{Q} &0\\ 
0 & 0
\end{array} \right) U^{*}. 
\end{align*}
Then $ A^{C ,\tddd}(A^{c,\dagger })^{\dagger}=(A^{c,\dagger })^{\dagger}A^{C, \tddd} $  if and only if  the  following conditions hold:
\begin{align}
Q^{*}\tilde{Q}(\hat{Q})^{\dagger}Q&=(\hat{Q})^{\dagger}\tilde{Q},\label{192}
\\Q^{*}\tilde{Q}(\hat{Q})^{\dagger} P&=0,\label{193}
\\ P^{*}\tilde{Q}(\hat{Q})^{\dagger}Q &= 0,\nonumber 
\\ P^{*}\tilde{Q}(\hat{Q})^{\dagger}P&=0.\label{195} 
\end{align}
Thus, by Theorem \ref{ADC}, we know that $ A^{c,\dagger } $ is EP if and only if 
\begin{align}
Q^{*}\hat{Q}(\hat{Q})^{\dagger} Q &=(\hat{Q})^{\dagger}\hat{Q},\label{196}
\\\hat{Q}(\hat{Q})^{\dagger}P&=0.\label{197}
\end{align}
By using $ (\Sigma Q)^{\tddd}=(\Sigma Q)^{d},$ the equations in (\ref{192}) and (\ref{196}) are equivalent.  By pre-multiplying (\ref{193}) by $Q $ and (\ref{195}) by $P $
and utilizing (\ref{KL}), we arrive at\\
$ \tilde{Q}(\hat{Q})^{\dagger}P = 0, $ equivalent to (\ref{197}).
\end{proof}
\section{Some properties of  CMP, DMP and MPD   inverses}
We start this section by considering characterizations and properties  of generalized inverses. 

In the below theorem, we describe $A_c $  by equations in (\ref{a101}).
\begin{theorem}
Let $ A \in M_{n}(\mathbb{C})$ with Ind$(A)=k.$ Then $ X= A_c$ is the unique solution of the following equations:
\begin{equation}\label{a101}
A^{k}X=A^{k+1}, \quad \quad AX=XA, \quad \quad XA^{d}X=X.
\end{equation}
\end{theorem}
\begin{proof}
It is evident that the matrix $ X= A_c $ fulfills the three equations in the system (\ref{a101}).
Now, we suppose that matrices $X_{1}$ and $ X_{2} $ satisfy (\ref{a101}). Then,
\begin{align*}
X_{1}&=X_{1}A^{d}X_{1}=X_{1}(A^{d})^{2}AX_{1}=X_{1}(A^{d})^{2}X_{1}A
=X_{1}(A^{d})^{k+2}A^{k}X_{1}A
\\&=X_{1}(A^{d})^{k+2}A^{k+1}A=X_{1}A^{k+1}(A^{d})^{k+2}A=A^{k}X_{1}A(A^{d})^{k+2}A
\\&=A^{k+1}A(A^{d})^{k+2}A=A^{k}X_{2}A(A^{d})^{k+2}A
=X_{2}A^{k+1}(A^{d})^{k+2}A
\\&=X_{2}(A^{d})^{k+2}A^{k+1}A
=X_{2}(A^{d})^{k+2}A^{k}X_{2}A=X_{2}(A^{d})^{2}X_{2}A
\\&=X_{2}(A^{d})^{2}AX_{2}=X_{2}A^{d}X_{2}=X_{2}.
\end{align*}
\end{proof}
\begin{proposition}\label{ehGH45}\cite[Theorem 3.2 ]{Ferreyra}
Let $ A \in M_{n}(\mathbb{C})$ with Ind$(A)= k$. Then 
\begin{enumerate}[label=(\roman*)]
\item
$A^{d,\dagger}=A_{\mathcal{R}(A^k) , \mathcal{N}(A^kA^\dagger)}^{(2)}$,
\item
$A^{\dagger , d}=A_{\mathcal{R}(A^{\dagger} A^k) , \mathcal{N}(A^k)}^{(2)}$.
\end{enumerate}
\end{proposition}
\begin{lemma}\label{ASS}
Let  $ A \in M_{n}(\mathbb{C})$ with Ind$(A)= k$. Then 
\begin{equation*}
A^{c,\dagger}=A^{\dagger ,d}A^{d,\dagger}~~~~~~ \Leftrightarrow \quad A^{k+1}=A^k ~~~\Leftrightarrow ~~~ \mathcal{R}(A^k) \subseteq \mathcal{N}(I- A).
\end{equation*}
\end{lemma}
\begin{proof}
 By Proposition \ref{ehGH45}$(i)$, we get 
\begin{align*}
A^{c,\dagger}=A^{\dagger ,d}A^{d,\dagger}&\Leftrightarrow A^{\dagger}AA^{d}AA^{\dagger}=A^{\dagger}AA^{d}A^{d}AA^{\dagger}\\
&\Leftrightarrow AA^{\dagger}AA^{d}AA^{\dagger}=AA^{\dagger}A^{d}AA^{\dagger}\\
&\Leftrightarrow AA^{d,\dagger}=A^{d,\dagger}\\
&\Leftrightarrow (I-A)A^{d,\dagger}=0\\
&\Leftrightarrow \mathcal{R}(A^{k})=\mathcal{R}(A^{d,\dagger}) \subseteq \mathcal{N}(I-A) \\
&\Leftrightarrow A^{k+1}=A^k. 
\end{align*}
\end{proof}

The following theorem gives the aforementioned relationships in terms of mainly the Moore-Penrose inverse.
\begin{theorem} 
Let  $ A \in M_{n}(\mathbb{C})$ with Ind$(A)= k$. Then 
\begin{enumerate}[label=(\roman*)]
\item
$(A^{d,\dagger})^{\dagger}A^{d}=A^{d}(A^{d,\dagger})^{\dagger} $~~~~~ $\Longleftrightarrow$ ~~~ $ (\Sigma Q)^{d} $ is EP and 
$ Q \Sigma P = 0. $
\item
$A^{c,\dagger}=A^{\dagger ,d}A $~~~~~~~~~~~~~~~~~~~~~~$\Longleftrightarrow$ ~~~ 
$A^kA^\dagger = A^k$.
\item
$A^{c,\dagger}=AA^{d, \dagger}$~~~~~~~~~~~~~~~~~~~~~~$\Longleftrightarrow$ ~~~ 
$A^\dagger A^k = A^k$.
\item
$A^{c,\dagger}=A^{\dagger ,d}A^{*} $~~~~~~~~~~~~~~~~~~~~$\Longleftrightarrow$ ~~~ 
$A^k(A^\dagger)^*=A^k$.
\item
$A^{c,\dagger}=A^{*}A^{d, \dagger} $~~~~~~~~~~~~~~~~~~~ $\Longleftrightarrow$ ~~~
$(A^\dagger)^* A^k=A^k$.
\end{enumerate}
\end{theorem}
\begin{proof}
$(i)$   By \cite[Proposition 2.15 (b)]{Malik}, (\ref{KL}) and (\ref{md}), we get
\begin{equation*}
(A^{d,\dagger})^{\dagger}A^{d}=U \left( \begin{array}{cc}
((\Sigma Q)^{d})^{\dagger}(\Sigma Q)^{d} &((\Sigma Q)^{d})^{\dagger}((\Sigma Q)^{d})^{2}\Sigma P \\ 
0 & 0
\end{array} \right) U^{*}, 
\end{equation*}
\begin{equation*}
A^{d}(A^{d,\dagger})^{\dagger}=U \left( \begin{array}{cc}
(\Sigma Q)^{d}((\Sigma Q)^{d})^{\dagger} & 0 \\ 
0 & 0
\end{array} \right) U^{*}. 
\end{equation*}
Therefore, $(A^{d,\dagger})^{\dagger}A^{d}=A^{d}(A^{d ,\dagger})^{\dagger}$ if and only if
\begin{equation*}
((\Sigma Q)^{d})^{\dagger}(\Sigma Q)^{d}=(\Sigma Q)^{d}((\Sigma Q)^{d})^{\dagger}\quad \& \quad ((\Sigma Q)^{d})^{\dagger}((\Sigma Q)^{d})^{2}\Sigma P=0.
\end{equation*}
The first equation states that $ (\Sigma Q)^{d} $ is EP (since it commutes with its Moore-Penrose inverse). Pre-multiplying the equation $ ((\Sigma Q)^{d})^{\dagger}((\Sigma Q)^{d})^{2}\Sigma P=0 $ by 
$ \Sigma \hat{Q} $ and using $ (\Sigma Q)^{d}((\Sigma Q)^{d})^{\dagger}(\Sigma Q)^{d}=(\Sigma Q)^{d},$ we obtain $ ( \Sigma Q )^{d} \Sigma P = 0. $ 
Since $(\Sigma Q)^d$ has index at most 1, it coincides with $(\Sigma Q)^\#$. So, the expression $(\Sigma Q)^d \Sigma P=0$ is equivalent to the more simplified one given by $Q \Sigma P=0$.

$ (ii) $ It is clear that 
$A^{c,\dagger}=A^{\dagger ,d}A \Leftrightarrow A^{\dagger}AA^{d}AA^{\dagger}=A^{\dagger}AA^{d}A$
$\Leftrightarrow$ $AA^{\dagger}AA^{d}AA^{\dagger}=AA^{\dagger}AA^{d}A$ 
$\Leftrightarrow$ $A_{c}A^{\dagger}=A_{c}$
$\Leftrightarrow$ $A_{c}(I-A^{\dagger})=0$
$\Leftrightarrow$ $\mathcal{R}(I-A^{\dagger}) \subseteq \mathcal{N}(A_{c})$.
Moreover,
\begin{align*}
\mathcal{N}(A_{c})&=\mathcal{N}(AA^{d}A)\subseteq \mathcal{N}(A^kA^{d}AA^{d}A)=\mathcal{N}(A^{k})
\subseteq \mathcal{N}((A^{d})^{k}A^{k})\\
&=\mathcal{N}(A^{d}A)\subseteq \mathcal{N}(A_{c}).
\end{align*}
Therefore, $ \mathcal{N}(A_{c})=\mathcal{N}(A^{k}) $. Now, we have that $ \mathcal{R}(I-A^{\dagger}) \subseteq \mathcal{N}(A^{k}) \Leftrightarrow A^kA^\dagger = A^k .$

$ (iii) $ It is similar to the proof of $(ii)$.

$ (iv) $ By Proposition \ref{ehGH45}$(ii)$, we obtain
\begin{align*}
A^{c,\dagger}=A^{\dagger ,d}A^{*}&\Leftrightarrow A^{\dagger}AA^{d}AA^{\dagger}=A^{\dagger}AA^{d}A^{*}\\
&\Leftrightarrow A^{\dagger}AA^{d}AA^{\dagger}(A^{\dagger})^{*}=A^{\dagger}AA^{d}A^{*}(A^{\dagger})^{*}\\
&\Leftrightarrow A^{\dagger}AA^{d}(A^{\dagger}AA^{\dagger})^{*}=A^{\dagger}AA^{d}(A^{\dagger}A)^{*}\\
&\Leftrightarrow A^{\dagger ,d}(A^{\dagger})^{*}=A^{\dagger ,d}\\
&\Leftrightarrow A^{\dagger ,d}(I-(A^{\dagger})^{*})=0\\
&\Leftrightarrow \mathcal{R}(I-(A^{\dagger})^{*}) \subseteq \mathcal{N}(A^{\dagger ,d})=\mathcal{N}(A^{k})\\
&\Leftrightarrow A^k(A^\dagger)^*=A^k.
\end{align*}
$ (v) $ By Proposition \ref{ehGH45}$(i)$ and similar to the proof of (iv).
\end{proof}

Item (i) in theorem above is equivalent to $ A^{d,\dag} $ is EP matrix and $Q \Sigma P=0$ \\
\cite[Proposition 2.15]{Malik}.

The following theorem gives the aforementioned relationships in terms of mainly the  core part of the matrix $A$.
\begin{theorem} 
Let $ A \in M_{n}(\mathbb{C})$ with Ind$(A)= k$. Then 
\begin{enumerate}[label=(\roman*)]
\item
$A^{d,\dagger}A_c=A_{c}A^{d,\dagger} $~~~~~~~~~ $\Longleftrightarrow$ ~~~ 
$\mathcal{N}(A^{*})\subseteq \mathcal{N}(A^k) $.
\item
$A^{\dagger ,d}A_c=A_{c}A^{\dagger ,d}$~~~~~~~~~~$\Longleftrightarrow$ ~~~ 
$\mathcal{R}(A^k)\subseteq \mathcal{R}(A^{*}) $.
\item
$A_c=A^{d, \dagger}A_c $~~~~~~~~~~~~~~~
 $\Longleftrightarrow$ ~~~ 
$A^k=A^{k+1}$.
\item
$A_c=A^{\dagger ,d}A_c $~~~~~~~~~~~~~~~$\Longleftrightarrow$ ~~~  
$A^{\dagger}A^{k} = A^k$.
\item
$ A_c=A^{c,\dagger}A_c $~~~~~~~~~~~~~~ 
$\Longleftrightarrow$ ~~~ 
$A{^\dagger}A^k=A^k$.
\end{enumerate}
\end{theorem}
\begin{proof}
$ (i) $ By  \cite[Remark 3.1]{Ferreyra}, we have 
\begin{align*}
A^{d,\dagger}A_c=A_{c}A^{d,\dagger}
&\Leftrightarrow 
A^{d}AA^{\dagger}AA^{d}A
=AA^{d}AA^{d}AA^{\dagger}\\
&\Leftrightarrow 
A^{d}A=AA^{d}AA^{\dagger}\\
&\Leftrightarrow A^{d}A(I-AA^{\dagger})=0\\
&\Leftrightarrow \mathcal{N}(A^{*})=\mathcal{N}(A^{\dagger})=\mathcal{N}(AA^{\dagger})=\mathcal{R}(I-AA^{\dagger}) \subseteq \mathcal{N}(A^{d}A)
\\&=\mathcal{N}(A^{d})=\mathcal{N}(A^k).
\end{align*}
Proofs of items $ (ii) $ and $ (iii) $ resemble to that of item $(i)$.

$(iv)$ 
\begin{align*}
A_c=A^{\dagger ,d}A_c&\Leftrightarrow 
A_c
=A^{\dagger}AA^{d}AA^{d}A\\
&\Leftrightarrow 
A_c
=A^{\dagger}A_c
\\&\Leftrightarrow (I-A^{\dagger})A_c=0
\\&\Leftrightarrow \mathcal{R}(A_c)\subseteq \mathcal{N}(I-A^{\dagger}).
\end{align*}
It is clear that $ \mathcal{R}(A_c)=\mathcal{R}(A^k) $.  We have  
$ \mathcal{R}(A^k)\subseteq \mathcal{N}(I-A^{\dagger}) \Leftrightarrow 
A^{\dagger}A^k = A^k$.

Part $ (v) $ is similar to the proof of $(iv)$.
\end{proof}
\begin{remark}
In order to compute explicitly DMP and MPD inverses are useful the following expressions:
\[
A^{d,\dagger} = A^k (A^{2k+1})^{\dagger }A^{k+1}A^{\dagger} \qquad \text{ and } \qquad 
A^{\dagger ,d} = A^{\dagger} A^{k+1} (A^{2k+1})^{\dagger} A^{k}, 
\] 
where $k=$Ind($A$). These formulas follow from the well-known Greville formula \\
$A^d=A^k(A^{2k+1})^{\dagger} A^k$, and they are interesting for computing both inverses by means of only the Moore-Penrose of some powers by using a package like MATLAB.

By using \cite[Corollary 3.8]{Ferreyra}, it is interesting compare the above one with the formula for the core-EP inverse of $A$ given by
\[
A^{\tddd} =A^{d}A^{k}(A^{k})^{\dagger}= A^k (A^{2k+1})^\dagger A^{2k}(A^k)^\dagger.
\]
In addition, by substracting both expressions, it is easy to see that $A^k(A^k)^\dagger=AA^\dagger$ implies $A^{d,\dagger} =A^{\tddd}$. 
\end{remark}
\begin{theorem}\label{q1} 
Let $ A \in M_{n}(\mathbb{C}). $ The general solution of equation 
$$XA= A^{\dagger}A_c$$ is given by 
$ X=A^{\dagger ,d}+F(I-AA^{\dagger})$, for arbitrary $F \in M_{n}(\mathbb{C})$.
\end{theorem}
\begin{proof}
By \cite[p. 52]{Ben}, we arrive at the general solution of $ XA=A^{\dagger}A_{c} $, which is given by
\begin{align*}
X&=A^{\dagger}A_{c}A^{\dagger}+Z-ZAA^{\dagger}\\
&=A^{\dagger}AA^{d}AA^{\dagger}+Z-ZAA^{\dagger}\\
&=A^{\dagger ,d}-A^{\dagger ,d}+Z-(Z-A^{\dagger ,d})AA^{\dagger}\\
&=A^{\dagger ,d}+(Z-A^{\dagger ,d})-(Z-A^{\dagger ,d})AA^{\dagger}\\
&=A^{\dagger ,d}+F(I-AA^{\dagger}),
\end{align*}
where 
$ F=Z-A^{\dagger ,d}$. 
\end{proof}
In a similar way, we prove the following result.
\begin{theorem}\label{q2}
Let $ A \in M_{n}(\mathbb{C}). $ The general solution of equation 
$$ 
A_cA^{\dagger}=AX
$$
is given by $ X=A^{d ,\dagger}+(I-A^{\dagger}A)F$, for arbitrary $F \in M_{n}(\mathbb{C})$. 
\end{theorem}
\begin{lemma}
Let $A \in M_{n}(\mathbb{C})$. Then $ X=A^{d}$ is a solution of the  following equation:
\begin{equation*}
XA^{\dagger ,d}=A^{d,\dagger}X.
\end{equation*}
\end{lemma}
\begin{proof}
Let $ X=A^{d}$. Then
\begin{align*}
A^{d}A^{\dagger ,d}&=A^{d}A^{\dagger}AA^{d}=(A^{d})^{2}AA^{\dagger}AA^{d}
\\&=(A^{d})^{2}AA^{d}=A^{d}A(A^{d})^{2}=A^{d}AA^{\dagger}A(A^{d})^{2}=A^{d,\dagger}A^{d}.
\end{align*}
\end{proof}
Note that, the relations in the last proof show that $$A^{d}A^{\dagger ,d}=A^{d,\dagger}A^{d}=(A^{d})^{2}.$$
\section{DMP and MPD binary relationships}
In this section, new binary relations based on the \emph{DMP} and \emph{MPD-}inverses are considered. 
The relationship between these binary relations and other binary relation orders is investigated.

Assume that  $ A, B \in M_{n}(\mathbb{C})$.
By \cite{MitraSK} and \cite[Definition 4.1]{new}, we state the following: 
\begin{align*}
A \qlestardod B ~~~~ & \textrm{if and only if} ~~~~ A^{d,\dagger}A=A^{d,\dagger}B\quad \&
\quad AA^{d,\dagger}=BA^{d,\dagger},\\
A \qlestardodd B ~~~~& \textrm{if and only if} ~~~~ A^{\dagger ,d}A=A^{\dagger ,d}B\quad \&
\quad AA^{\dagger ,d}=BA^{\dagger ,d},\\
A \qlestardot B ~~~~ & \textrm{if and only if} ~~~~ A^{c,\dagger}A=A^{c,\dagger}B\quad \& 
\quad AA^{c,\dagger}=BA^{c,\dagger}, \\
A \qlestardooD B ~~~~ & \textrm{if and only if} ~~~~ A^{d}A=A^{d}B\quad \& 
\quad AA^{d}=BA^{d}.
\end{align*}

Next results shows that the core part of a matrix $A$ is always an upper bound of $A$ under the considered binary relations. 
\begin{theorem}\label{EW2}
Let $ A \in M_{n}(\mathbb{C}) $. Then 
\begin{enumerate}[label=(\roman*)]
\item
$ A \qlestardot A_{c}$,
\item
$ A \qlestardooD A_{c}$, 
\item
$ A \qlestardodd A_{c}$,
\item
$ A \qlestardod A_{c}$.
\end{enumerate}
\end{theorem}
\begin{proof}
$ (i) $  By  \cite[Theorem 2.1]{am16}, we have
\begin{align*}
A^{c,\dagger}A&=A^{\dagger}AA^{d}AA^{\dagger}A=A^{\dagger}AA^{d}AA^{d}A
=A^{\dagger}AA^{d}AA^{\dagger}AA^{d}A=A^{c,\dagger}A_{c}\\
AA^{c,\dagger}&=AA^{\dagger}AA^{d}AA^{\dagger}=AA^{d}AA^{d}AA^{\dagger}=
AA^{d}AA^{\dagger}AA^{d}AA^{\dagger}=A_{c}A^{c,\dagger}.
\end{align*}
Proofs of items $ (ii) $, $ (iii) $ and $ (iv) $ are similar to that of item (i).
\end{proof}
\begin{theorem}\label{ADF}
Let $ A, B \in M_{n}(\mathbb{C}) $ with Ind$(A)= k$.
Then the following are equivalent:
\begin{enumerate}[label=(\roman*)]
\item
$ A \qlestardod B$,
\item
$ A^{d}=A^{d}A^{\dagger}B=B(A^{d})^{2}$, 
\item
$ A^{k}=A^{k}A^{\dagger}B=BA^{d}A^{k}$.
\end{enumerate}
\end{theorem}
\begin{proof}
$(i)\Rightarrow (ii)$ If $ A \qlestardod B$, then $A^{d,\dagger}A=A^{d,\dagger}B$ and 
$ AA^{d,\dagger}=BA^{d,\dagger}.$ Thus
\begin{align*}
A^{d,\dagger}A=A^{d,\dagger}B&\Leftrightarrow A^{d}AA^{\dagger}A=A^{d}AA^{\dagger}B\\
&\Leftrightarrow A^{d}A=A^{d}AA^{\dagger}B\\
&\Leftrightarrow A^{d}A^{d}A=A^{d}A^{d}AA^{\dagger}B\\
&\Leftrightarrow A^{d}=A^{d}A^{\dagger}B.
\end{align*}
Similarly, $ AA^{d,\dagger}=BA^{d,\dagger} \Leftrightarrow A^{d}=B(A^{d})^{2}. $

$(ii) \Rightarrow (iii)$ It is trivial.

$(iii) \Rightarrow (i) $ Let $ A^{k}=A^{k}A^{\dagger}B $ and $ A^{k}=BA^{d}A^{k}$. Then 
\begin{align*}
A^{k}=A^{k}A^{\dagger}B&\Leftrightarrow (A^{d})^{k}A^{k}=(A^{d})^{k}A^{k}A^{\dagger}B\\
&\Leftrightarrow AA^{d}=A^{d}AA^{\dagger}B\\
&\Leftrightarrow A^{d}AA^{\dagger}A=A^{d}AA^{\dagger}B\\
&\Leftrightarrow A^{d,\dagger}A=A^{d,\dagger}B.
\end{align*}
Similarly, $ A^{k}=BA^{d}A^{k} \Leftrightarrow AA^{d,\dagger}=BA^{d,\dagger}$.
\end{proof}
The following theorem is derived by using the same techique as in Theorem \ref{ADF}.
\begin{theorem}
Assume that $ A, B \in M_{n}(\mathbb{C}) $ with Ind$(A)= k$. 
Then the following are equivalent:
\begin{enumerate}[label=(\roman*)]
\item
$ A \qlestardodd B$,
\item
$ A^{d}=(A^{d})^{2}B=BA^{\dagger}A^{d}$, 
\item
$ A^{k}=A^{k}A^{d}B=BA^{\dagger}A^{k}$.
\end{enumerate} 
\end{theorem}
\begin{remark}\label{DU2}
Assume that $ A \in M_{n}(\mathbb{C}) $ with Ind$(A)= k$.
By \cite[Proposition 3.3]{KS} and \cite[Theorem 3.3 and Theorem 3.5]{am16}, we arrive at the conclusion that 
$A$ is a k-EP matrix  (that is, $A^k A^\dagger = A^\dagger A^k$)  if and only if
$ A^{c\dagger}=A^{d,\dagger}=A^{\dagger ,d}=A^{d}$. 
\end{remark}
By Remark \ref{DU2}, and \cite[Proposition 4.7]{new}, 
we have the following remark.
\begin{remark}
Let $ A \in M_{n}(\mathbb{C}) $ with Ind$(A)= k$. If $A$ is a $k$-EP, then the following four binary relations are  equivalent:
$A \qlestardodd B,$
$ A \qlestardod B,$
$A \qlestardot B,$
$ A \qlestardooD B.$
\end{remark}
\begin{example}\label{ex1}
\begin{small}
Let $ A = \left( \begin{array}{ccc}
2 & 0 & 0 \\ 
0 & 0 & 0 \\ 
2 & 2 & 0 
\end{array} \right)$, $ B= \left( \begin{array}{ccc}
2 & 0 & 0 \\ 
0 & 0 & 0 \\ 
1 & 0 & 1 
\end{array} \right)$. 
\end{small}
Then, Ind($A$)$=2$, 
\begin{equation*}
A^{d}=A^{d,\dagger} = \left( \begin{array}{ccc}
\frac{1}{2} & 0 & 0 \\ 
0 & 0 & 0 \\
\frac{1}{2} & 0 & 0 
\end{array} \right), \quad
A^{c, \dagger}=A^{\dagger ,d} = \left( \begin{array}{ccc}
\frac{1}{2} & 0 & 0 \\ 
0 & 0 & 0 \\
0 & 0 & 0 
\end{array} \right).
\end{equation*}
It is readily seen that $ A \qlestardooD B $, $ A \qlestardod B$, but  $ A \nqlestardodd B$ and $ A \nqlestardot B$.
\end{example}

\section*{Acknowledgement}
The authors would like to thank the anonymous referees for their
careful reading and their valuable comments and suggestions that help
us to improve the reading of the paper.

\section*{Funding} 

The third author was partially supported by Universidad Nacional de La
Pampa (Argentina) Facultad de Ingenier\'ia [Grant Resol. Nro. 135/19]
and Ministerio de Ciencia, Innovaci\'on y Universidades (Spain) [Grant
Redes de Investigaci\'on, MICINN-RED2022-134176-T].

\section*{Declarations}
{\bf Conflict of interest}  There is no conflict of interest in the manuscript.
 \bibliographystyle{spbasic} 
\bibliography{reference}

\end{document}